\documentclass[reqno, intlimits, sumlimits]{amsart} 

\usepackage{amsmath, amstext, amssymb, amsthm, bbm, lmodern, enumerate, esint}

\usepackage[left=3cm,right=3cm, top=4cm, bottom=3cm]{geometry}

\usepackage{hyperref}
\hypersetup{colorlinks=true,linkcolor=black,citecolor=black,filecolor=black,urlcolor=black}

    \renewcommand{\phi}{\varphi}
    \renewcommand{\epsilon}{\varepsilon}

	\newcommand				{\eins}			{\mathbbm{1}}   
	
	\newcommand				{\abs}[1]		{\left\lvert#1\right\rvert}

	\DeclareMathOperator	{\IE}			{\mathbb{E}} 
	
	\DeclareMathOperator	{\IP}			{\mathbb{P}}
	
	\DeclareMathOperator	{\IR}			{\mathbb{R}}

	\DeclareMathOperator	{\supp}			{supp}

	\theoremstyle{plain}
\newtheorem{thm}			{Theorem}[section]
\newtheorem{lem}	[thm]	{Lemma}
\newtheorem{cor}	[thm]	{Corollary}
\newtheorem{prop}	[thm]	{Proposition}

\theoremstyle{definition}

\newtheorem{rem}	[thm]	{Remark}

\numberwithin{equation}{section}

\begin{document}

\title[Transportation of random measures not charging small sets]{Transportation of random measures not charging small sets}


\author[M. Huesmann]{Martin Huesmann}
\address{M.H.: Universit\"at M\"unster, Germany}
\email{martin.huesmann@uni-muenster.de}
\author[B. Müller]{Bastian Müller}
\address{B.M.: Universit\"at M\"unster, Germany   }
\email{bastian.mueller@uni-muenster.de}

\thanks{MH and BM are funded by the Deutsche Forschungsgemeinschaft (DFG, German Research Foundation) under Germany's Excellence Strategy EXC 2044 -390685587, Mathematics M\"unster: Dynamics--Geometry--Structure and   by the DFG through the SPP 2265 {\it Random Geometric Systems. }}

\begin{abstract}
Let $(\xi,\eta)$ be a pair of jointly stationary, ergodic random measures of equal finite intensity. A balancing allocation is a translation-invariant (equivariant) map $T:\IR^d\to\IR^d$ such that the image measure of $\xi$ under $T$ is $\eta$. We show that as soon as $\xi$ does not charge small sets, i.e.\  does not give mass to $(d-1)$-rectifiable sets, there is always a balancing allocation $T$ which is measurably depending only on $(\xi,\eta)$, i.e. $T$ is a factor.
\end{abstract}

\date{\today}
\maketitle

\section{Introduction}

Let $\xi$ and $\eta$ be two random, jointly stationary, and ergodic measures with the same finite intensity. An allocation $T$ is a translation-invariant (equivariant) random mapping $T:\IR^d \to \IR^d$. It is said to balance $\xi$ and $\eta$, if the image measure of $\xi$ under $T$ is equal to $\eta$. In this article, we are interested in the question of existence of balancing allocations. Note that without the requirement of translation-invariance existence can be shown via Borel isomorphism theorems as soon as $\xi$ is diffuse, i.e.\ $\xi$ does not have atoms. However, the requirement of translation-invariance makes the question much harder.
Last and Thorisson showed recently the following existence result:

\begin{thm}[{\cite[Theorem 1.1]{LastThorisson}}]\label{thm:LT}
Let $\xi$ and $\eta$ be two random, jointly stationary, and ergodic measures with the same finite intensity. Let $\xi$ be diffuse.
Then there exists an allocation balancing $\xi$ and $\eta$ if one of the following conditions holds:

\begin{enumerate}[(a)]
\item $\eta$ has a non-zero discrete component;
\item $\eta$ is diffuse and there exists a non-zero simple point process $\chi$ on $\IR^d$ with finite intensity, such that the triple $(\xi,\eta,\chi)$ is jointly stationary and ergodic.
\end{enumerate}
\end{thm}

The most interesting applications (see below) of part (b) are in the case when the process $\chi$ is derived as a factor of $(\xi,\eta)$, i.e.\ if it is measurably dependent on $(\xi,\eta)$.  Without the requirement of being a factor such a process $\chi$ can be constructed by an extension of the probability space. This raises the question of either characterizing pairs of random measures $(\xi,\eta)$ admitting a point process factor $\chi$ or deriving complementary conditions ensuring existence of balancing factor allocations, e.g.\ see \cite{Haji_Mirsadeghi_2016, LastThorisson}.
In this article, we concentrate on the latter question and derive conditions only on $\xi$ ensuring the existence of allocations. By the example of  \cite[Section 8]{LastThorisson}, we know that for such a general existence result $\xi$ should not give mass to $d-1$-dimensional sets.
Indeed,  Last and Thorisson constructed a pair of jointly stationary, ergodic, diffuse random measures $(\xi,\eta)$, where $\xi$ is concentrated on a $d-1$ dimensional set, such that there is no balancing allocation.  Our main result gives a general existence result for factor allocations:

\begin{thm}\label{thm:main}
Let $\xi$ and $\eta$ be two random, jointly stationary, and ergodic measures with the same finite intensity. Assume that $\xi$ does not charge small sets, i.e. does not give mass to $(d-1)$-rectifiable sets. Then there exists a factor allocation balancing $\xi$ and $\eta$.
\end{thm}

We note  that the assumption on $\xi$ is sharp  by the counterexample of Last and Thorisson.  We also remark that Theorem \ref{thm:main} (just as Theorem \ref{thm:LT}) remains true if one relaxes the assumption of ergodicity and same intensity to the assumption that $\IE[\xi([0,1])| \mathcal I]=\IE[\eta([0,1])| \mathcal I]$, where $\mathcal I$ denotes the $\sigma$-algebra of shift invariant events (e.g. see \cite[Section 9]{LastThorisson}).

The proof of Theorem \ref{thm:main} is based on the optimal transport techniques for random measures, introduced in \cite{HS13} and \cite{H16}. Let us sketch the argument. By  \cite[Theorem 5.1]{Last09}, under our assumptions there is some equivariant coupling $q$ for $\xi$ and $\eta$. By an application of the Lemma of de la Vall\'ee Poussin, we can construct a concave function $\vartheta$, such that $q$ has finite mean transportation cost (cf.\ \eqref{eq:meancost}) w.r.t. $c(x,y)=\vartheta(|x-y|)$. If $\xi$ and $\eta$ are mutually singular, \cite[Theorem 1.1]{H16} implies the existence of an equivariant coupling $q^*=(id,T)_\#\xi$. In particular, $T$ is the desired balancing factor allocation.

For the general case, we first construct an auxiliary factor allocation $T$ between the mutually singular measures $(\xi-\eta)_+$ and $(\eta-\xi)_+$. This can be used to partition $\IR^d$ in an equivariant way into sets $\{x:G(T(x)-x)\leq t\}$ and $\{x:G(T(x)-x)> t\}$, where $G:\IR^d\to \IR^d$ is some fixed deterministic map. It turns out that for a particular choice of $t=t_0$ the measures $\eins_{G(T(x)-x)\leq t_0}\xi$ and $ \eins_{G(T(x)-x)> t_0}\eta$ have the same intensity (and are mutually singular). Hence, there exists a balancing factor allocation $T_1$. Similarly, we obtain a balancing factor allocation $T_2$ between $\eins_{G(T(x)-x)> t_0}\xi$ and $ \eins_{G(T(x)-x)\leq t_0}\eta$. Combining $T_1$ and $T_2$ proves Theorem \ref{thm:main}.

The interest in allocations orginates from its link to shift couplings of random measures with their Palm version. If the source $\xi$ is the  Lebesgue measure on $\IR^d$ and  $T$ an allocation balancing $\mathsf{Leb}$ and $\eta$,  then the shifted measure $\eta -T(0)$ is the Palm version of $\eta$, i.e.\ the pair $(\eta,\eta-T(0))$ is a shift coupling of $\eta$ and its Palm version (see \cite{Holroyd_2005}). In particular, if $T$ is a factor allocation, the Palm version of $\eta$ is a function of $\eta$. To the best of our knowledge, the first explicit non-randomized (factor) shift-coupling for point processes was constructed by Liggett \cite{Liggett2002}. This work together with \cite{Holroyd_2005,Hoffman_2006} initiated a series of constructions of shift couplings by constructing factor allocations, e.g.\  
\cite{Holroyd_2005, Chatterjee, LaMoTh14, HS13}. Allocations and equivariant couplings or transports between two general random measures $\xi$ and $\eta$ have been investigated e.g.\ in 
\cite{Last09, LaMoTh14, Last_2018, H16}.
 We also refer to \cite{AlTh93, Thorisson} for the origin of shift-couplings and to \cite{Last09} and \cite[Remark 2.2]{LastThorisson} for results on shift couplings resulting from allocations between general random measures.
 
\begin{rem}
As  a particular consequence of the preceding paragraphs and Theorem \ref{thm:main}, for any $\eta$ there is always a (factor) shift coupling of $\eta$ with its Palm version (by \cite{Thorisson} we only know that there is some shift coupling on a potentially enlarged probability space).
\end{rem}

\section{Setup and Preliminaries}\label{sec:setup}
Let $(\Omega,\mathcal{F},\IP)$ be a probability space  equipped with a measurable flow $\theta_x:\Omega\to \Omega$, $x\in \IR^d$. That is, the mapping $(x,\omega)\mapsto \theta_x\omega$ is measurable, $\theta_y\circ \theta_x=\theta_{x+y}$ for all $x,y\in \IR^d$ and $\theta_0$ is the identity.
Furthermore, let  $\IP$  be stationary  w.r.t. the flow $\theta$, i.e. $\IP(A)=\IP(\theta_x(A))$ for all $x\in \IR^d$.
The invariant sigma field $\mathcal{I}$ is defined by $\mathcal{I}=\{A\in \mathcal{F}\mid \forall x\in \IR^d: A=\theta_xA \}$ and we assume that $\IP$ is ergodic, that is $\IP(A)\in \{0,1\}$ for all $A\in \mathcal{I}$.

In the following a random measures $\xi$ is a locally finite transition kernel from $(\Omega,\mathcal{F})$ to $(\IR^d,\mathcal{B}(\IR^d))$, where locally finite means, that for  $\IP$-a.e. $\omega \in \Omega$  the measure $\xi(\omega,\cdot)$ is finite on bounded measurable sets.
A random measure $\xi$ is said to be equivariant if for all $\omega \in \Omega$, $x\in \IR^d$ and $B\in \mathcal{B}(\IR^d)$ it holds that \[
\xi(\omega,B)=\xi(\theta_x\omega,B-x).
\]

An allocation is a measurable mapping $T:\Omega\times \IR^d\to \IR^d$ with the following equivariance property \[
T(\theta_x\omega,y)=T(\omega,y+x)-x \quad \forall \omega \in \Omega, x,y\in \IR^d.
\]
The allocation $T$ balances two random measures $\xi$ and $\eta$ if for $\IP$-a.e. $\omega\in \Omega$ the map $T^{\omega}$ pushes $\xi^{\omega}$ onto $\eta^{\omega}$, i.e. $\xi^{\omega}\circ (T^{\omega})^{-1}=\eta^{\omega}$. We say that $T$ is a factor allocation, if $T$ is measurable w.r.t. to $\sigma(\xi,\eta)$, the sigma algebra generated by $\xi$ and $\eta$.

A  semicoupling $q$ of $\xi$ and $\eta$ is a  transition kernel  from $(\Omega,\mathcal{F})$ to $(\IR^d\times \IR^d,\mathcal{B}(\IR^d\times \IR^d))$ such that for every $\omega
\in \Omega$ the measure $q(\omega)$ is a semicoupling of $\xi(\omega)$ and $\eta(\omega)$, that is \begin{align}\label{eq:semicoupl}
(\pi_1)_{\#}(q(\omega))\leq \xi(\omega) \text{ and }(\pi_2)_{\#}(q(\omega))=\eta(\omega).
\end{align}
Here $\pi_i$ dentotes the projection onto the $i$-th coordinate. A semicoupling $q$ is said to be equivariant if \[
q(\omega,A\times B)=q(\theta_x\omega,(A-x)\times (B-x))\quad \forall \omega\in\Omega, x\in \IR^d, A,B\in \mathcal{B}(\IR^d).
\]
For equivariant random measures $\xi$ and $\eta$ we denote by $\mathsf{Cpl}_{es}(\xi,\eta)$ the set of all equivariant semicouplings of $\xi$ and $\eta$. 
 For a given function $\vartheta:[0,\infty)\to [0,\infty)$ we then define the mean transportation cost by 
 \begin{equation}\label{eq:meancost}
\inf_{q\in \mathsf{Cpl}_{es}(\xi,\eta)}\IE\left[\int_{\Lambda_1\times \IR^d}\vartheta(\abs{x-y})q(dx,dy)\right].
\end{equation}

Optimal transport problems for semicoupling between finite measures have been also investigated under the name of partial optimal transport problem, e.g.\ \cite{Fi10}, or incomplete optimal transportation \cite{AEBaCAMa11}. We will establish a particular  uniqueness result for concave cost for a partial optimal transport between finite measures in Lemma \ref{lem:unique_semicp} below. It is an important ingredient for the proof of the following  theorem:

\begin{thm}[Semicoupling]\label{thm:semicoupling}
Let $\xi$ and $\eta$ be two equivariant random measures, which are a.s. mutually singular.
 Furthermore, assume that a.s. $\xi$ does not charge small sets, i.e. $\xi$ does not give mass to $(d-1)$-rectifiable sets and that the intensity of $\xi$ is greater than or equal to the intensity of $\eta$. Let $\vartheta$ be a strictly increasing, concave function $\vartheta:[0,\infty)\to[0,\infty)$ with $\lim_{x\to \infty}\vartheta(x)=\infty$ and $\vartheta(0)=0$. Assume that the mean transportation cost of $\xi$ and $\eta$ w.r.t.  $\vartheta$ is finite. Then there exists an equivariant semicoupling $q$ of $\xi$ and $\eta$, which can be represented as $q=(Id,T)_{\#}(f\cdot\xi)$, for some allocation $T:\supp(\xi)\to \supp(\eta)$ and $f:\IR^d\to [0,\infty)$, measurably only dependent on $\sigma(\xi,\eta)$.
 
 Moreover, if $\xi$ and $\eta$ have equal intensities, then the equivariant semicoupling is in fact a coupling.
\end{thm}

Theorem \ref{thm:semicoupling} can be proven exactly as \cite[Theorem 1.1]{H16} once we have established the following uniqueness result for semicouplings, where we denote by $\mathsf{Cpl}_s(\mu,\nu)$ the set of semicouplings between $\mu$ and $\nu$.

\begin{lem}\label{lem:unique_semicp}
Let $\mu,\nu$ be two finite Borel measures on $\IR^d$ such that $\mu(\IR^d)\ge \nu(\IR^d)$, $\mu$ does not charge small sets and $\mu$ and $\nu$ are mutually singular. Let $c(x,y)=\vartheta(|x-y|)$ for some concave strictly increasing function $\vartheta:\IR_+\to\IR_+$. Then there is a unique optimizer $q*$ to 
\begin{align}\label{eq:opt}
\inf_{q \in \mathsf{Cpl}_{s}(\mu,\nu)} \int c(x,y) dq(x,y) .
\end{align}
Moreover, $q*=(id,T)(1_{B}\mu)$ for some measurable set $B$ and a map $T$. 
\end{lem}

\begin{proof}
By compactness of the support of $\mu,\nu$, it follows that the set of all semicouplings between $\mu$ and $\nu$ is compact. Since $c$ is continuous and bounded from below it follows that the map $q\mapsto \int c(x.y) dq(x,y)$ is lower semicontinuous. Hence, there exists a minimizer $q*$ with marginals $f\cdot\mu$ and $\nu$. Moreover, since $q*$ is optimal between its marginals, it follows by mutual singularity of $\mu$ and $\nu$ that there exists a map $T$ such that $q*=(id,T)(f\cdot \mu)$, \cite[Theorem 4.6]{Sant15}. 

We claim that $f=1_{B}$ for some measurable set $B$. This implies uniqueness. Indeed, if $q_{1}$ and $q_{2}$ are two potentially different optimizer with densities $1_{B_{1}}$ and $1_{B_{2}}$ respectively, then also $q_{3}=\frac12(q_{1}+q_{2})$ is an optimizer by linearity whose density has to satisfy $1_{B_{3}}=\frac12(1_{B_{1}}+1_{B_{2}}) \ \mu-a.s.$ Hence, we obtain that $B_{1}=B_{2}=B_{3}\ \mu-a.s.$ and therefore uniqueness.

To show the claim, we will argue by contradiction. Let us assume that $\mu(\{ f<1\})>0$ so that there is an $\varepsilon> 0$ such that $A=\{f\le 1-\varepsilon\}$ has positive $\mu$ measure. Then, $\tilde q=q_{|A\times \IR^d}$ is optimal between its marginals $\tilde f\mu$ and $\tilde \nu$. For notational simplicity we can then assume that $\tilde q=q$ and $\tilde f=f\le 1-\varepsilon.$ By mutual singularity of $\mu$ and $\nu$ there is $y$ such that 
\begin{itemize}
\item[i)] $\nu(B_{r}(y))>0$ for all $r>0$, i.e.\ $y \in \mathsf{supp}(\nu)$
\item[ii)] $\lim_{r\to 0}\frac{\mu(B_{r}(y))}{\nu(B_{r}(y))} = 0$ by mutual singularity of $\mu$ and $\nu$.
\end{itemize}
In particular, for any $\delta>0$ there is $r>0$ such that $\nu(B_{r}(y))>0$ and $\mu(B_{r}(y))\le \delta \nu(B_{r}(y))$ such that $q(B_{r}(y)^c\times B_{r}(y))>0$. However, since $f\le 1-\varepsilon$ we can use the the mass within $T^{-1}(B_{r}(y))\setminus B_{r}(y)$ which is transported to $B_{r}(y)$ more efficiently to produce a coupling with cheaper cost. In the remaining part of the proof, we will explicitly construct such a competitor to $q$.

Let $(x_0,y_0)\in \mathsf{supp}(q_{|B_r(y)^c\times B_r(y)})$ and choose $r'>0$
sufficiently small, for example $r'<\sqrt{\frac{1}{{2d}}|x_0-y_0|}$ suffices.
  Since $q(B_{r'}(x_0)\times B_{r'}(y_0))>0$ there exists by \cite[Lemma 4.1]{Sant15}
 a point  $x_1\in B_{r'}(x_0)$ with the following property \begin{align}\label{eq:GMTlemma}
\forall \alpha>0,\forall \delta>0,\forall u\in \mathbb{S}^{d-1}: q((C(x_1,u,\delta,\alpha)\cap B_{r'}(x_0))\times B_{r'}(y_0))>0,
\end{align}where
\[
C(x_1,u,\delta,\alpha)=\{z:u\cdot (z-x_1)\geq (1-\delta)|z-x_1|\}\cap \bar{B}_{\alpha}(x_1).
\]
Let  $u_-=\frac{x_1-y_0}{|x_1-y_0|}$ and $u_+=\frac{y_0-x_1}{|y_0-x_1|}$ and set $C_-=C(x_1,u_-,\delta,\alpha)\setminus\{x_1\}$ and $C_+=C(x_1,u_+,\delta,\alpha)\setminus\{x_1\}$.
Then an elementary geometric argument shows that the following holds. For all $\delta, \alpha>0$ small enough\begin{align}\label{eq:cone_inequ}
\forall z_- \in C_-, \forall z_+ \in C_+, \forall \tilde{y}\in B_{r'}(y_0): |z_+-\tilde{y}|<|z_--\tilde{y}|.
\end{align}
In the following fix such $r',\delta, \alpha >0$. In particular, let $r'>\alpha$ so that the intersection in \eqref{eq:GMTlemma} reduces to \[
C_+\cap B_{r'}(x_0)=C_+,
\]and similiarly for $C_-$. For $0<s<1$ there exists $t=t(s)>1$ such that \[
sq(C_-\times B_{r'}(y_0))+tq(C_+\times B_{r'}(y_0))=q((C_-\cup C_+)\times B_{r'}(y_0)).
\]
Let $\pi$ be an optimal coupling of $(t-1)\mathsf{pr}_1(q_{|C_+ \times B_{r'}(y_0)})$ and $(1-s)\mathsf{pr}_2(q_{|C_- \times B_{r'}(y_0)})$. These measures have the same mass since\[
sq(C_-\times B_{r'}(y_0))+tq(C_+\times B_{r'}(y_0))=q((C_-\cup C_+)\times B_{r'}(y_0))
=q(C_-\times B_{r'}(y_0))+q( C_+\times B_{r'}(y_0)).
\]
Then define \[
\hat q=sq_{|C_-\times B_{r'}(y_0)}+q_{|C_+\times B_{r'}(y_0)}+\pi.
\]
Since $\lim_{s\nearrow 1}t(s)=1$, it follows from  $f\leq 1-\epsilon$ that $t f\leq 1$ for $s<1$ large enough.
 Hence $\hat q$ defines a semicoupling of $\mu_{|C_-\cup C_+}$ and $\mathsf{pr}_2(q_{|(C_-\cup C_+)\times B_{r'}(y_0)})$. Thus $\hat q$ is an admissible competitor to $q_{|(C_-\cup C_+)\times B_{r'}(y_0)}$. Disintegration w.r.t. the second marginal of the measures $q_{|(C_-\cup C_+)\times B_{r'}(y_0)}$ and $\pi$ yields \begin{align*}
&\int_{(C_-\cup C_+)\times B_{r'}(y_0)} \vartheta(|x-y|)dq-\int_{(C_-\cup C_+)\times B_{r'}(y_0)} \vartheta(|x-y|)d\hat q\\
&=(1-s)\int_{C_-\times B_{r'}(y_0)} \vartheta(|x-y|)dq-\int_{ C_+\times B_{r'}(y_0)} \vartheta(|x-y|)d\pi\\
&=(1-s)\int_{B_{r'}(y_0)}\mathsf{pr}_2(q_{|C_- \times B_{r'}(y_0)})(d\tilde{y})\left[\int_{C_-} dq_{\tilde{y}}(dx)\vartheta(|x-y|)-\int_{C_+}d\pi_{\tilde{y}}(dx)\vartheta(|x-y|)\right].
\end{align*}
Inequality \eqref{eq:cone_inequ} implies that the last line is strictly positive. That is,
\[
\int_{(C_-\cup C_+)\times B_{r'}(y_0)} \vartheta(|x-y|)dq>\int_{(C_-\cup C_+)\times B_{r'}(y_0)} \vartheta(|x-y|)d\hat q.
\]
This, however, contradicts the optimality of $q$.
\end{proof}

We give a very short sketch of the proof of Theorem \ref{thm:semicoupling}.

\begin{proof}[Sketch of proof of Theorem \ref{thm:semicoupling}]

Existence of an optimal semicoupling can be proven exactly as in the proof of \cite[Proposition 3.18]{H16}. 

In order to establish uniqueness, we introduce the notion of local optimality.
In our setup, an equivariant coupling $q$ is locally optimal iff the following holds for $\IP$-a.e. $\omega\in\Omega$:

There exists a nonnegative density $\rho^{\omega}$ and a c-cyclically monotone map $T^{\omega}:\{\rho^{\omega}>0\}\to \IR^d$ such that on $\{\rho^{\omega}>0\}\times \IR^d$\[
q^{\omega}=(Id,T^{\omega})_{\#}(\rho^{\omega}\xi^{\omega});
\]
see Definition 5.3 in \cite{H16}.
Local optimality of optimal semicouplings can be shown as in Proposition 3.1 and Theorem 3.6 in \cite{HS13}. The proof of \cite[Proposition 3.1]{HS13} relies on uniqueness of optimal semicouplings on bounded sets, a fact which in our setting is provided by Lemma \ref{lem:unique_semicp}. 

Now we can prove uniqueness of optimal semicouplings. Let $q_1,q_2$ be two optimal semicouplings. By local optimality there exist maps $T_i$ and densities $\rho_i$, $i=1,2$, such that 
$q_i^{\omega}=(Id,T_i^{\omega})_{\#}(\rho_i^{\omega}\xi^{\omega})$. Restricting the $q_i$ to bounded sets, it follows from optimality and from Lemma \ref{lem:unique_semicp}, that we can assume $\rho_i^{\omega}=\eins_{A_i^{\omega}}$, for some measurable set $A_i^{\omega}\subset \IR^d$. Applying the same reasoning to the optimal semicoupling $q=\frac{1}{2}(q_1+q_2)$ proves that $q_1=q_2$.
\end{proof}

Finally we will need the following version of the Lemma of de la Vall\'ee Poussin.

\begin{lem}\label{lem_Vallee_Poussin}
Let $f:[0,\infty)\to [0,\infty)$ be a Lebesgue integrable function. Then there exists a continous,  concave function $\vartheta:[0,\infty)\to[0,\infty)$  with $\vartheta(0)=0$ such that \[
\int_0^{\infty} f(x) \vartheta(x)dx<\infty.
\] 
Moreover $\vartheta$ can be chosen to be strictly increasing, smooth on $(0,\infty)$ and such that  $\lim_{x\to \infty}\vartheta(x)=\infty$.
\end{lem}
\begin{proof}
Let $\Phi$ be the convex function obtained by applying \cite[Theorem 2.8]{Laurencot} to the  function $g(x)=x$  and the measure $\eins_{[1,\infty)}(x)\frac{f(x)}{x}dx$. That is, $\Phi\in C^{\infty}([0,\infty))$, $\Phi'(0)=\Phi(0)=0$, $\Phi$ is integrable w.r.t. $\eins_{[1,\infty)}(x)\frac{f(x)}{x}dx$ and $\Phi'$ is a concave function. Moreover \[
\frac{\Phi(r)}{r}\xrightarrow{x\to \infty}\infty\text{ and } \Phi'(r)>0\quad \forall  r>0.
\]Define $\vartheta(x)=\frac{\Phi(x)}{x}$ for $x>0$ and extend this function continuously by letting $\vartheta(0)=\Phi'(0)=0$. Note that, since $\Phi$ is convex, $\vartheta$ is strictly increasing. Concavity of $\vartheta$ follows from \cite[Proposition 2.14]{Laurencot}. Finally the integrability of $f\cdot \vartheta$ follows from \begin{align*}
\int_0^{\infty} f(x) \vartheta(x)dx&=\int_0^1f(x)\vartheta(x)dx+\int_1^{\infty}f(x)\frac{\Phi(x)}{x}dx\\
&\leq \sup_{x\in [0,1]}\vartheta(x)\int_0^1f(x)dx+\int_1^{\infty}f(x)\frac{\Phi(x)}{x}dx<\infty.
\end{align*}
\end{proof}


\section{Proof of Theorem \ref{thm:main}}

From now on, we will assume that $(\xi,\eta)$ are jointly stationary and ergodic random measures with the same finite intensities. We start by showing that there is a concave strictly increasing and diverging function $\vartheta$ such that the mean transportation cost \eqref{eq:meancost} w.r.t.\ $c(x,y)=\vartheta(|x-y|)$ is finite. Combining this with Theorem \ref{thm:semicoupling} implies existence of allocations in the case that $\xi$ and $\eta$ are mutually singular, see Subsection \ref{sec:mutually_sing}. In a next step we will prove our main result in the case that both $\xi$ and $\eta$ do not charge small sets, see Subsection \ref{sec:smallsets}. Finally, we will show the general statement in Subsection \ref{sec:general}.

\subsection{Existence of an equivariant coupling with finite cost}

\begin{lem}\label{lem:cost_function}
Let $\xi$ and $\eta$ be two jointly stationary and ergodic random measures with the same finite intensity.
There exists an equivariant coupling $q$ of $\xi$ and $\eta$ and a concave function $\vartheta:[0,\infty)\to[0,\infty)$ such that \[
\IE\left[\int_{\Lambda_1\times \IR^d}\vartheta(\abs{x-y})q(dx,dy)\right]<\infty.
\]Furthermore, the function $\vartheta$ can be chosen to be continuous and strictly increasing and such that $\vartheta(0)=0$ and $\lim_{x\to \infty}\vartheta(x)=\infty$.
\end{lem}
\begin{proof}
By \cite[Theorem 5.1]{Last09}, there exists an equivariant coupling $q$ of $\xi$ and $\eta$, since their intensities coincide. We are going to construct the desired function $\vartheta$. Since $\xi$ has finite intensity, we can write
\begin{align*}
\IE\left[\xi\left(\Lambda_1\right)\right]=\IE\left[\int_{\Lambda_1\times \IR^d}q(dx,dy)\right]=\sum_{n\geq 0}\IE\left[\int_{\Lambda_1\times \IR^d}\eins_{n\leq \abs{x-y}<n+1}q(dx,dy)\right]=\sum_{n\geq 0}a_n<\infty,
\end{align*}
with $a_n=\IE\left[\int_{\Lambda_1\times \IR^d}\eins_{n\leq \abs{x-y}<n+1}q(dx,dy)\right]$.
Define the function $f:[0,\infty)\to[0,\infty)$ by $f=\sum_{n\geq 0}\eins_{[n,n+1)}a_n$. By construction, $f$ is integrable. Hence, from Lemma \ref{lem_Vallee_Poussin} it follows that there exists a function $\vartheta$ with the properties listed in the statement of this lemma, which is smooth on $(0,\infty)$ and satifies 
\[
\sum_{n\geq 0}a_n\int_n^{n+1}\vartheta(x)dx=\int_0^{\infty}\vartheta(x)f(x)dx<\infty.
\]Then \begin{align*}
\IE\left[\int_{\Lambda_1\times \IR^d}\vartheta(\abs{x-y})q(dx,dy)\right]&=
\sum_{n\geq 0}\IE\left[\int_{\Lambda_1\times \IR^d}\vartheta(\abs{x-y})\eins_{n\leq \abs{x-y}<n+1}q(dx,dy)\right]\\
&\leq \sum_{n\geq 0}\vartheta(n+1)\IE\left[\int_{\Lambda_1\times \IR^d}\eins_{n\leq \abs{x-y}<n+1}q(dx,dy)\right]\\
&= \sum_{n\geq 0}\vartheta(n+1)a_n.
\end{align*}
For $n\geq 1$ we estimate using concavity of $\vartheta$ in the last step
\[
\vartheta(n+1)-\int_n^{n+1}\vartheta(x)dx=\int_n^{n+1}\vartheta(n+1)-\vartheta(x)dx \leq \sup_{x\in [1,\infty)}\vartheta'(x)=\vartheta'(1).
\] Hence we can bound \begin{align*}
&\IE\left[\int_{\Lambda_1\times \IR^d}\vartheta(\abs{x-y})q(dx,dy)\right] 
\leq a_0\vartheta(1)+\sum_{n\geq 1}\vartheta(n+1)a_n\\
&\leq a_0\vartheta(1)+\sum_{n\geq 1}a_n \left(\int_n^{n+1}\vartheta(x)dx+\vartheta'(1)\right) <\infty.
\end{align*}
\end{proof}

\subsection{Mutually singular measures}\label{sec:mutually_sing}

\begin{cor}\label{cor:mutually_singular}
Let $\xi$ and $\eta$ be two jointly stationary and ergodic random measures with the same finite intensity, which are  $a.s.$ mutually singular. Furthermore, assume that $\xi$ does not charge $(d-1)$-rectifiable sets. Then there exists a factor allocation.
\end{cor}
\begin{proof}
From  Lemma \ref{lem:cost_function} we obtain a function $\vartheta$, which yields finite  mean transportation cost and satisfies the properties listed in Theorem \ref{thm:semicoupling}.  The other assumptions of Theorem \ref{thm:semicoupling} are also satisfied. Finally note that, since the random measures $\xi$ and $\eta$ have the same intensity, the optimal semicoupling is a coupling. Hence the random map $T^{\omega}$ is  a factor allocation for $\xi$ and $\eta$. 
\end{proof}

\subsection{Measures that do not charge small sets}\label{sec:smallsets}
In this subsection, we assume that both $\xi$ and $\eta$ do not charge small sets. We consider the decompositions $\xi=(\xi\wedge \eta)+(\xi-\eta)_+$ and  $\eta=(\xi\wedge \eta)+(\eta-\xi)_+$.
Here $(\xi-\eta)_+$ denotes the positive part of the Jordan decomposition of $\xi-\eta$ and the measure $(\eta-\xi)_+$ is analogously defined.
 Note that  the measures $(\xi-\eta)_+$ and $(\eta-\xi)_+$ are mutually singular, do not charge small sets, i.e.\ do not give mass to $(d-1)$-rectifiable sets, and have the same intensity. 
 
\begin{prop}\label{prop:no_small_sets}
Let $\xi$ and $\eta$ be two jointly stationary and ergodic random measures with the same finite intensity. Assume that $\xi$ and $\eta$ do not charge small sets. Then there exists a factor allocation.
\end{prop}
\begin{proof}
Let $T:\supp(\xi-\eta)_+\to \supp(\eta-\xi)_+$ be the factor allocation for the mutually singular measures, which exists by Corollary \ref{cor:mutually_singular}. Since both measures do not charge small sets, there exists also the inverse allocation $T^{-1}:\supp(\eta-\xi)_+\to \supp(\xi-\eta)_+$.
We define the (random) function $F$ on $\IR^d$ by 
    \[ F(x) = \begin{cases} 
          T(x) & x\in \supp(\xi-\eta)_+ \\
          x & \text{otherwise} \\
          T^{-1}(x) & x\in \supp(\eta-\xi)_+ 
       \end{cases}
    \]
Let $G:\IR^d\to \IR$ be a measurable and bijective function such that $G(0)=0$.    
Define the function $I:\IR\to \IR$ by
\[
I(t)=\IE\left[\int_{\Lambda_1}\eins_{G(F(x)-x)\leq t}\eins_{F(x)-x\neq 0}\xi(dx)\right].
\]
This function  satifies $\lim_{t\to -\infty}I(t)=0$ and $\lim_{t\to \infty}I(t)=\IE\left[\int_{\Lambda_1}\eins_{F(x)-x\neq 0}\xi(dx)\right]>0$. We prove that it is continuous. Since it is increasing in $t$, it suffices to prove that for fixed $t\in \IR$ \[
\IE\left[\int_{\Lambda_1}\eins_{G(F(x)-x)= t}\eins_{F(x)-x\neq 0}\xi(dx)\right]=0.
\]
This is true for $t=0$ so let $t\neq 0$. Then 
\begin{align}\label{eq:zero_expectations}
&\IE\left[\int_{\Lambda_1}\eins_{G(F(x)-x)= t}\eins_{F(x)-x\neq 0}\xi(dx)\right]
=\IE\left[\int_{\Lambda_1}\eins_{F(x)-x=G^{-1}(t)}\xi(dx)\right]\nonumber\\
&=\IE\left[\int_{\Lambda_1\cap\supp(\xi-\eta)_+}\eins_{T(x)-x=G^{-1}(t)}\xi(dx)\right]
+ \IE\left[\int_{\Lambda_1\cap\supp(\eta-\xi)_+}\eins_{T^{-1}(x)-x=G^{-1}(t)}\xi(dx)\right].
\end{align}
Now consider a fixed realisation of the measure $ \eins_{\Lambda_1\cap\supp(\xi-\eta)_+}\xi$ and of the corresponding pushforward $T_{\#}(\eins_{\Lambda_1\cap\supp(\xi-\eta)_+}\xi)$. From  \cite[Theorem 5.5]{H16} it follows that $T$ is  an optimal transport map for the measures 
$ \eins_{\Lambda_1\cap\supp(\xi-\eta)_+}\xi$ and $T_{\#}(\eins_{\Lambda_1\cap\supp(\xi-\eta)_+}\xi)$ w.r.t. the cost $c(x,y)=\vartheta(\abs{x-y})$.
Applying \cite[Proposition 5.1]{Sant15} thus yields that a.s. \[
\xi\left(\{x\in \Lambda_1\cap\supp(\eta-\xi)_+: T(x)-x=G^{-1}(t)\}\right)=0.
\]
Hence the first expectation is zero. Since we restrict to the set $\supp(\eta-\xi)_+$, we can bound the second expectation in \eqref{eq:zero_expectations} in the following way from above   \[
\IE\left[\int_{\Lambda_1\cap\supp(\eta-\xi)_+}\eins_{T^{-1}(x)-x=G^{-1}(t)}\xi(dx)\right]
\leq \IE\left[\int_{\Lambda_1\cap\supp(\eta-\xi)_+}\eins_{T^{-1}(x)-x=G^{-1}(t)}\eta(dx)\right].
\] 
By the same argument we used for the first expectation, it follows that the upper bound is equal to zero. Hence both terms in \eqref{eq:zero_expectations} are equal to zero and the continuity is proved.
 We define the corresponding function $J(t)$ by
\[
J(t)=\IE\left[\int_{\Lambda_1}\eins_{G(F(x)-x)> t}\eins_{F(x)-x\neq 0}\eta(dx)\right].
\]
This function is continuous as well and has the limits  
$\lim_{t\to -\infty}J(t)=\IE\left[\int_{\Lambda_1}\eins_{F(x)-x\neq 0}\eta(dx)\right]>0$ and $\lim_{t\to \infty}J(t)=0$. Hence there exists a $t_0$ such that $I(t_0)=J(t_0)$.

 This means that the random measures $ \eins_{G(F(x)-x)\leq t_0}\eins_{F(x)-x\neq 0}\xi$ and $  \eins_{G(F(x)-x)> t_0}\eins_{F(x)-x\neq 0}\eta$ have the same intensity. Since they are mutually singular, we can apply Corollary \ref{cor:mutually_singular} to obtain a factor allocation \[S_1:\supp( \eins_{G(F(x)-x)\leq t_0}\eins_{F(x)-x\neq 0}\xi)\to \supp(  \eins_{G(F(x)-x)> t_0}\eins_{F(x)-x\neq 0}\eta).
\]
Similiar arguments yield a factor allocation $S_2$ for the measures 
$$ \eins_{G(F(x)-x)> t_0}\eins_{F(x)-x\neq 0}\xi \text{ and }   \eins_{G(F(x)-x)\leq t_0}\eins_{F(x)-x\neq 0}\eta.$$
 Defining $S_3:\supp( \eins_{F(x)-x=0}\xi)\to \supp( \eins_{F(x)-x=0}\eta)$ to be the identity map, we see that $T=S_1+S_2+S_3$ is a factor allocation for the measures $\xi$ and $\eta$.
\end{proof}

\subsection{General case}\label{sec:general}
Combining Corollary \ref{cor:mutually_singular} and Proposition \ref{prop:no_small_sets} we prove the most general case. 
\begin{thm}
Let $\xi$ and $\eta$ be two jointly stationary and ergodic random measures with the same finite intensity. Assume that $\xi$ does not charge small sets. Then there exists a factor allocation.
\end{thm}
\begin{proof}
Via the Lebesgue decomposition theorem we can write $\eta=\eta^a+\eta^s$, where $\eta^a$ is absolutely continuous w.r.t. $\xi$. The measures  $\eta^s$ and $\xi$ are mutually singular.

By Lemma \ref{lem:cost_function} there exists a function $\vartheta$ and an equivariant coupling $\tilde{q}$ of $\xi$ and $\eta$ s.t. \[
\IE\left[\int_{\Lambda_1\times \IR^d}\vartheta(\abs{x-y})\tilde{q}(dx,dy)\right]<\infty.
\]
Note that restricting $\tilde{q}$ to the (random) set $\IR^d\times \supp(\eta^s)$ yields a semicoupling of $\xi$ and $\eta^s$. This semicoupling has finite mean transportation cost w.r.t.\ the function $\vartheta$, since the cost is bounded by the above expectation.
Since $\xi$ and $\eta^s$ are mutually singular, Theorem \ref{thm:semicoupling} yields an equivariant semicoupling $q$ of $\xi$ and $\eta^s$.

Denote by $\tilde{\xi}$ the first marginal of $q$ and by $f$ the density of $\tilde{\xi}$ w.r.t. $\xi$. Theorem \ref{thm:semicoupling} also yields a factor allocation $T$, which pushes $\tilde{\xi}$ onto $\eta^s$.
Note that the intensity of the random measure $\eins_{f>0}\xi$ is greater or equal than the intensity of the measure $\eins_{f>0}\tilde{\xi}=\tilde{\xi}$, which coincides with the intensity of $\eta^s$.
Similiar to the previous proof define $I(t)$ with the same deterministic function $G$ by 
\[
I(t)=\IE\left[\int_{\Lambda_1}\eins_{f(x)>0} \eins_{G(T(x)-x)\leq t}\xi(dx)\right].
\]
We see that $\lim_{t\to -\infty }I(t)=0$ and that $\lim_{t\to \infty }I(t)=\IE\left[\int_{\Lambda_1}\eins_{f(x)>0}\xi(dx)\right]
\geq \IE\left[\int_{\Lambda_1}d\eta^s\right]$.

Continuity of $I$ follows similiarly to the previous proof.  Let $t\in \IR$. Then 
\begin{align*}
&\IE\left[\int_{\Lambda_1}\eins_{f(x)>0} \eins_{G(T(x)-x)= t}\xi(dx)\right]=\IE\left[\int_{\Lambda_1}\frac{1}{f(x)} \eins_{G(T(x)-x)= t}\tilde{\xi}(dx)\right]\\
&=\lim_{N\to \infty}\IE\left[\int_{\Lambda_1}\left(\frac{1}{f(x)}\wedge N\right) \eins_{G(T(x)-x)= t}\tilde{\xi}(dx)\right]\leq \lim_{N\to \infty}N\IE\left[\int_{\Lambda_1} \eins_{G(T(x)-x)= t}\tilde{\xi}(dx)\right].
\end{align*}
Since by  \cite[Theorem 5.5]{H16} (local optimality, see sketch of proof of Theorem \ref{thm:semicoupling}) $T$ is an optimal map for the transport between the measures $\eins_{\Lambda_1}\eins_{G(T(x)-x)= t}\tilde{\xi}$ and $T_{\#}(\eins_{\Lambda_1}\eins_{G(T(x)-x)= t}\tilde{\xi})$, we can  apply again \cite[Proposition 5.1]{Sant15}.
Hence for every $N$ the integral inside the expectation in the last line is a.s.\ equal to $0$. This proves the continuity.

By the intermediate value theorem  there exists a $t_0\in \IR$ such that the measures $ \eins_{f>0}\eins_{G(T(x)-x)\leq t_0}\xi$ and $\eta^s$ have the same intensity. Define the random set 
\[
A=\{x\in \IR^d:f(x)>0\}\cap \{x\in \IR^d:G(T(x)-x)\leq t_0\}.
\]
Since the measures $ \eins_A\xi$ and $\eta^s$ have the same intensity and are mutually singular, there exists a factor allocation $S_1$ by Corollary \ref{cor:mutually_singular}. 

Note that the measures $ \eins_{A^c}\xi$ and $\eta^a$ do not charge small sets, because $\eta^a$ is absolutely continuous w.r.t. $\xi$. Hence, by Proposition \ref{prop:no_small_sets}, there exists a factor allocation $S_2$, which pushes  $ \eins_{A^c}\xi$ onto $\eta^a$. The map $T=S_1+S_2$ is the desired factor allocation.
\end{proof}




\bibliographystyle{alpha} 
\bibliography{bibliography}

\begin{thebibliography}{AEdBCAM11}

\bibitem[AEdBCAM11]{AEBaCAMa11}
P.~C. \'Alvarez-Esteban, E.~del Barrio, J.~A. Cuesta-Albertos, and C.~Matr\'an.
\newblock Uniqueness and approximate computation of optimal incomplete
  transportation plans.
\newblock {\em Annales de l'I.H.P. Probabilit\'es et statistiques},
  47(2):358--375, 2011.

\bibitem[AT93]{AlTh93}
David~J. Aldous and Hermann Thorisson.
\newblock Shift-coupling.
\newblock {\em Stochastic Processes and their Applications}, 44(1):1--14, 1993.

\bibitem[CPPR10]{Chatterjee}
Sourav Chatterjee, Ron Peled, Yuval Peres, and Dan Romik.
\newblock Gravitational allocation to {P}oisson points.
\newblock {\em Ann. of Math. (2)}, 172(1):617--671, 2010.

\bibitem[Fig10]{Fi10}
Alessio Figalli.
\newblock The optimal partial transport problem.
\newblock {\em Archive for Rational Mechanics and Analysis}, 195(2):533--560,
  2010.

\bibitem[HHP06]{Hoffman_2006}
Christopher Hoffman, Alexander~E. Holroyd, and Yuval Peres.
\newblock A stable marriage of {P}oisson and {L}ebesgue.
\newblock {\em Ann. Probab.}, 34(4):1241--1272, 2006.

\bibitem[HMK16]{Haji_Mirsadeghi_2016}
Mir-Omid Haji-Mirsadeghi and Ali Khezeli.
\newblock Stable transports between stationary random measures.
\newblock {\em Electron. J. Probab.}, 21:Paper No. 51, 25, 2016.

\bibitem[HP05]{Holroyd_2005}
Alexander~E. Holroyd and Yuval Peres.
\newblock Extra heads and invariant allocations.
\newblock {\em Ann. Probab.}, 33(1):31--52, 2005.

\bibitem[HS13]{HS13}
Martin Huesmann and Karl-Theodor Sturm.
\newblock Optimal transport from {L}ebesgue to {P}oisson.
\newblock {\em Ann. Probab.}, 41(4):2426--2478, 2013.

\bibitem[Hue16]{H16}
Martin Huesmann.
\newblock Optimal transport between random measures.
\newblock {\em Ann. Inst. Henri Poincar\'{e} Probab. Stat.}, 52(1):196--232,
  2016.

\bibitem[Lau15]{Laurencot}
Philippe Lauren\c{c}ot.
\newblock Weak compactness techniques and coagulation equations.
\newblock In {\em Evolutionary equations with applications in natural
  sciences}, volume 2126 of {\em Lecture Notes in Math.}, pages 199--253.
  Springer, Cham, 2015.

\bibitem[Lig02]{Liggett2002}
Thomas~M. Liggett.
\newblock Tagged particle distributions or how to choose a head at random.
\newblock In {\em In and out of equilibrium ({M}ambucaba, 2000)}, volume~51 of
  {\em Progr. Probab.}, pages 133--162. Birkh\"{a}user Boston, Boston, MA,
  2002.

\bibitem[LMT14]{LaMoTh14}
G\"{u}nter Last, Peter M\"{o}rters, and Hermann Thorisson.
\newblock Unbiased shifts of {B}rownian motion.
\newblock {\em Ann. Probab.}, 42(2):431--463, 2014.

\bibitem[LT09]{Last09}
G\"{u}nter Last and Hermann Thorisson.
\newblock Invariant transports of stationary random measures and
  mass-stationarity.
\newblock {\em Ann. Probab.}, 37(2):790--813, 2009.

\bibitem[LT21]{LastThorisson}
G{\"u}nter Last and Hermann Thorisson.
\newblock Transportation of diffuse random measures on {$\mathbb R^{d}$}.
\newblock {\em arXiv:2112.13053}, 2021.

\bibitem[LTT18]{Last_2018}
G\"{u}nter Last, Wenpin Tang, and Hermann Thorisson.
\newblock Transporting random measures on the line and embedding excursions
  into {B}rownian motion.
\newblock {\em Ann. Inst. Henri Poincar\'{e} Probab. Stat.}, 54(4):2286--2303,
  2018.

\bibitem[PSP15]{Sant15}
Paul Pegon, Filippo Santambrogio, and Davide Piazzoli.
\newblock Full characterization of optimal transport plans for concave costs.
\newblock {\em Discrete Contin. Dyn. Syst.}, 35(12):6113--6132, 2015.

\bibitem[Tho96]{Thorisson}
Hermann Thorisson.
\newblock Transforming random elements and shifting random fields.
\newblock {\em Ann. Probab.}, 24(4):2057--2064, 1996.

\end{thebibliography}
\end{document}